\documentclass[12pt, a4paper, twoside]{amsart}

\usepackage[T1]{fontenc}
\usepackage[utf8]{inputenc}
\usepackage{amsmath}
\usepackage{amsthm}
\usepackage{amsfonts}
\usepackage{bbm}
\usepackage[UKenglish]{babel}
\usepackage{enumerate}
\usepackage{bm}
\usepackage{ae}
\usepackage{amssymb}
\usepackage[draft]{fixme}
\usepackage{pgf,tikz}
\usepackage{url}

\theoremstyle{definition}
\newtheorem{defn}{Definition}

\theoremstyle{remark}

\theoremstyle{plain}
\newtheorem{thm}{Theorem}
\newtheorem{lem}[thm]{Lemma}

\makeatletter
\newcommand*{\house}[1]{%
  \mathord{%
    \mathpalette\@house{#1}%
  }%
}
\newcommand*{\@house}[2]{%
  \dimen@=\fontdimen8 %
      \ifx#1\scriptscriptstyle\scriptscriptfont
      \else\ifx#1\scriptstyle\scriptfont
      \else\textfont\fi\fi
      3 %
  \sbox0{%
    $#1%
      \vrule width\dimen@\relax
      \overline{%
        \kern2\dimen@
        \begingroup 
          #2%
        \endgroup
        \kern2\dimen@
      }%
      \vrule width\dimen@\relax
      \mathsurround=1.5\dimen@ 
    $%
  }%
  \ht0=\dimexpr\ht0-\dimen@\relax
  \dp0=\dimexpr\dp0+2\dimen@\relax
  \vbox{%
    \kern\dimen@ 
    \copy0 %
  }%
}
\makeatother

\DeclareMathOperator{\BBP}{BBP}
\DeclareMathOperator{\CTB}{CTB}

\subjclass{Primary 11K16; Secondary 11Y60}
\keywords{BBP-type formulas; null formulas; formulas for $\pi$}

\begin{document}

\title{BBP-type formulas --- an elementary approach}

\author{SIMON KRISTENSEN}

\address{S. Kristensen, Department of Mathematics, Aarhus University, Ny Munkegade 118,
  DK-8000 Aarhus C, Denmark}

\email{sik@math.au.dk}

\author{OSKAR MATHIASEN}

\address{O. Mathiasen, Department of Mathematics, Aarhus University, Ny Munkegade 118,
  DK-8000 Aarhus C, Denmark}

\email{201907609@post.au.dk}

\begin{abstract}
We provide a simple way of searching for formulas of the Bailey--Borwein--Plouffe type together with an algorithm and an implementation in \texttt{sage}. Aside from rediscovering some already known formulas, the method has been used in the discovery of a new BBP-type formula for $\sqrt{3}\pi$. In addition, the implementation is very flexible and allows us to look for BBP-type formulas to irrational bases but with integer coefficients. As an example of this, searching in various Pisot bases, we have discovered a formula for $\pi$ in base $1 + \sqrt{3}$, along with additional formulas.
\end{abstract}

\maketitle

\section{Introduction}

In 1997, D. Bailey, P. B. Borwein and S. Plouffe \cite{MR1415794} discovered the formula
\begin{equation}\label{eq:BBP}
\pi = \sum_{k=0}^\infty \frac{1}{16^k} \left(\frac{4}{8k+1} + \frac{-2}{8k+4}+\frac{-1}{8k+5} + \frac{-1}{8k+6}\right)
\end{equation}
The novelty of the formula is that it allows one to calculate -- using a spigot algorithm -- the $n$'th digit in the hexadecimal expansion of $\pi$ without knowing the preceeding digits. Since then, many other formulas of similar type have been discovered. We refer the reader to D. Bayley's on-line compendium \cite{bailey} for a list of some of the known formulas.

Before proceeding, we will need some notation. 

\begin{defn}
A BBP-type formula is a series of the form 
\[
\BBP(d,b,n,A) = \sum_{k=0}^\infty \frac{1}{b^k} \sum_{j=1}^n \frac{a_j}{(kn+j)^d},
\]
where $d, b, n \in \mathbb{N}$ are called the degree, base and number respectively, and $A = (a_1, \dots, a_n) \in \mathbb{R}^n$ is a vector.
\end{defn}

Often, one requires the $a_j$'s to be integers, and clearly this is the more interesting case. However, we will not be making this restriction in calculations, although we will present only results with integer coefficients. In this way, \eqref{eq:BBP} states that
\[
\pi = \BBP(1,16,8, (4,0,0,-2,-1,-1,0,0)).
\]

A unified approach to searching for BBP-type formulas was however sadly lacking until D. Barsky, V. Muñoz and R. Pérez-Marco \cite{MR4246921} provided a framework for searching for such formulas. Their idea is to set up a family of iterated integrals, depending on a single complex variable, $s$ say. These integrals are subsequently expanded into rapidly converging series. On substituting particular values of $s$ into the series and considering real and imaginary parts separately, one obtains BBP-type formulas for a variety of numbers. 

A main feature of their work is that they find representations of zero as a BBP-type formula. These so-called null formulas have been among the more mysterious ones in literature, and have given rise to relations between BBP-type formulas.

In this paper, we give a less general form of the main theoretical tool of Barsky, Muñoz and Pérez-Marco \cite{MR4246921}, which as it turns out admits a simple and intuitive geometric interpretation. In addition, it is rather simple to implement an algorithm based on the main result, which is flexible enough to deal with a number of interesting cases.

\section{Main theoretical results}

We present here our main theoretical result, which in its simplicity provides several BBP-type formulas.

\begin{thm}
\label{thm:main}
Let $\theta = \frac{a}{b}2\pi$ with $\frac{a}{b} \in \mathbb{Q}$ and $2 \vert b$ and let $r \in (0,1)$. Then,
\[
\arg(1+r \cdot e^{i\theta}) = \BBP\left(1,\tfrac{1}{r^b}, b, A\right),
\]
where $A = (a_1, \dots, a_b)$ with $a_j = r^j (-1)^{j+1} \sin(j\theta)$.
\end{thm}

\begin{proof}
Our starting point is the Taylor series for $\log(1+z)$ centred at $z=0$,
\[
\log(1+z) = \sum_{n=1}^\infty (-1)^{n-1} \frac{z^n}{n}.
\]
Substituting $z = r e^{i\theta}$, we get,
\[
\log(1+re^{i\theta}) = \sum_{n=1}^\infty (-1)^{n-1} \frac{r^n e^{in\theta}}{n}.
\]
Recalling that the argument of a complex number is nothing but the imaginary part of the logarithm of the number, 
\[
\arg(1+re^{i\theta}) = \sum_{n=1}^\infty (-1)^{n-1} \frac{r^n \sin(n\theta)}{n}.
\]
Now, $\sin(n\theta)$ is clearly $b$-periodic, so writing each $n = bk+j$ with $j=1,2,\dots, b$, we get since $2 \vert b$,
\[
\arg(1+re^{i\theta}) = \sum_{j=0}^\infty \left(\frac{1}{r^{b}}\right)^{-k} \sum_{j=1}^b \frac{r^j (-1)^{j+1}\sin(j\theta)}{bk+j},
\]
as required.
\end{proof}

Note that we could just as well have taken the real part in the proof to obtain a BBP-type formula for $\log \vert 1+r \cdot e^{i\theta} \vert = \frac{1}{2} \log(1+r^2+ 2r\cos \theta)$. In the example given below, this would have resulted in a BBP-type formula for $\log 2$. The same procedure gives a formula for the logarithm of a certain real number for each of our obtained formulas below. We have left it for the interested reader to deduce such formulas from the new ones obtained in this paper.

As an example of our main theorem, let us derive a BBP-type formula for $\pi/4$, originally found by Ferguson, Bailey and Arno \cite{MR1489971}. Letting $r = \tfrac{1}{\sqrt{2}}$, $a=3$ and $b = 8$, our theorem states that
\[
\frac{\pi}{4} = \arg\left(1+\tfrac{1}{\sqrt{2}} e^{\frac{3}{8}2\pi i}\right) = \BBP(1,16,8,A) = \sum_{k=0}^\infty \frac{1}{16^k} \sum_{j=1}^8 \frac{a_j}{8k+j},
\]
where $A = (a_1,\dots, a_8) = (\frac12, \frac12, \frac14,0, \frac{-1}8, \frac{-1}8, \frac{-1}{16},0)$.

We can interpret the main result geometrically in the following way. Consider the circle in $\mathbb{C}$ with centre $1$ and radius $r$, along with the triangle subtended by the three points $A = 0$, $B=1$ and $C = 1+ re^{i\theta}$. Clearly, $C$ lies on the circle considered. Moreover, the angle at $B$ is $\pi - 2\theta \pi$, a rational multiple of $\pi$. If also the angle at $A$ is a rational multiple of $\pi$, the theorem yields a BBP-type formula for $\pi$. In the example given, the triangle is an isosceles triangle with angles at $A$ and $B$ equal to $\frac{\pi}4$, see figure \ref{fig:triangle1}.

\begin{figure}
\scalebox{.7}{\begin{tikzpicture}
\draw[->] (-1.1*3,0) -- (4*3,0);
\draw[->] (0,-2*3) -- (0,2*3);
\draw (2*3,0) circle [radius=1.4142*3];
\draw[red, thick] (0,0) -- (3, 3) -- (6,0) -- (0,0);
\filldraw[fill=green!20,draw=green!50!black] (0,0) -- (.6,0)
      arc [start angle=0, end angle=45, radius=.6] node[anchor= west] {$\tfrac{\pi}4$} -- cycle ;
\filldraw[fill=green!20,draw=green!50!black] (6,0) -- (1.8*3,0)
      arc [start angle=180, end angle=135, radius=.6] node[anchor= east] {$\tfrac{\pi}4$}-- cycle;
\draw[fill=black] (0,0) circle (2pt) node[anchor=north east] {$A = 0$};
\draw[fill=black] (3,3) circle (2pt) node[anchor=west] {$C = 1+ \tfrac{1}{\sqrt{2}}e^{\frac{3}8 2 \pi i}$};
\draw[fill=black] (6,0) circle (2pt) node[anchor=north west] {$B=1$};
\end{tikzpicture}}
\caption{Proof that $\frac{\pi}{4} = \BBP(1,16,8,(\frac12, \frac12, \frac14,0, \frac{-1}8, \frac{-1}8, \frac{-1}{16},0))$}
\label{fig:triangle1}
\end{figure}
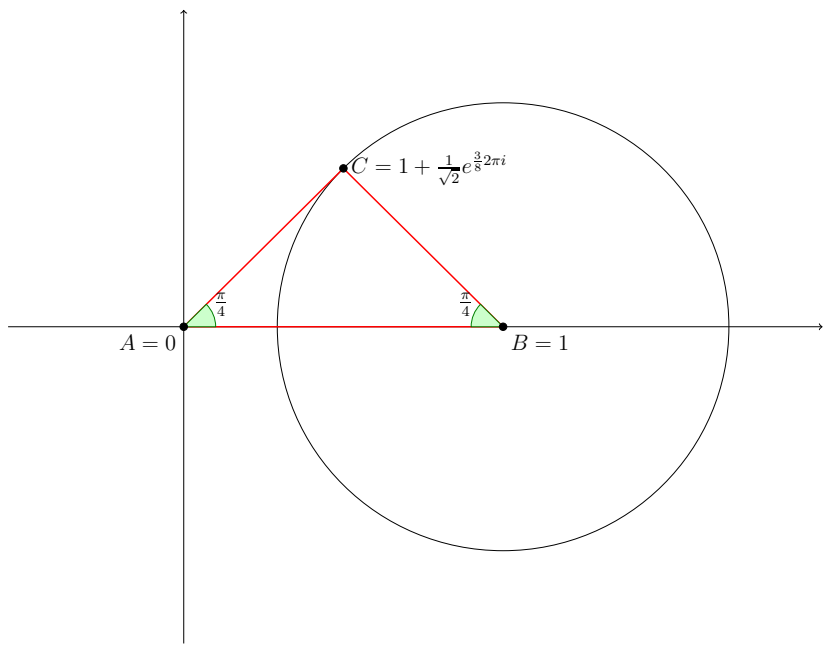

Of course, such triangles are rather uncommon. However, as we will see, we can sometimes combine triangles, which are not necessarily as nice, and obtain new BBP-type formulas. The observation which makes this work is the following lemma.

\begin{lem}
For fixed $d,n \in \mathbb{N}$ and $b > 1$, the map $\BBP(d,b,n, \cdot): \mathbb{R}^n \rightarrow \mathbb{R}$ is linear.
\end{lem}

The proof is simply a matter of manipulating convergent series and hence ommitted.

We would like to consider all BBP-type formulas arising from the main theorem and look for combinations of them, which give formulas of interest. Recall that the formulas arising in this way are for the angle of the $ABC$-triangle at $A$, so unless this number is of interest, the formulas obtained are not interesting in themselves, although it may happen that they combine with other formulas to reveal a formula of interest. In order to study these combinations, we define a function which given an even natural number $b$ we define the Circle-To-BBP map $\CTB_b: (0,1) \times \mathbb{N}\rightarrow \mathbb{R}^b$ by letting
\[
\CTB_b(r,a) = A = (a_1, \dots, a_b), \quad \text{with } a_j = r^j (-1)^{j+1} \sin\left(\tfrac{ja}{b}2\pi\right).
\]
The map reads off the vector associated to the BBP-type formula arising from Theorem \ref{thm:main} with the parameters $r$, $b$ and $\theta = \frac{a}b 2\pi$. In other words, the range of the map is the collection of permissible vectors from triangles arising by letting the point $C$ vary over all corners of a regular $b$-gon inscribed in a circle with centre $1$ and radius $r$. 

By periodicity of sine, $\CTB_b(r, \cdot)$ is clearly $b$-periodic, but in fact one only needs to consider $a \in \{1, \dots, \frac{b}2-1\}$. Indeed, using that sine is odd and $2\pi$-periodic, we immediately verify the relations in the following lemma.

\begin{lem}
For $b$ an even natural number and $r \in (0,1)$,
\begin{enumerate}
\item $\CTB_b(r,0)=\CTB_b(r, \frac{b}2)= 0$.
\item $\CTB_b(r, a) = - \CTB_b(r, b-a)$.
\end{enumerate}
\end{lem}

\section{An algorithm}

At this point, we are ready to turn the above results into an algorithm. We first explain how our algorithm works when searching for formulas for $0$. We will subsequently explain how to modify it to look for formulas for other constants. In this case, algorithm works as follows.

\begin{enumerate}
\item Fix an algebraic base number $r \in (0,1)$ (strictly speaking, the base will be $r^{-1}$) and an even natural number $b$.
\item Compute all vectors $\CTB_b(r,a)$ for $a \in \{1, \dots, \frac{b}2-1\}$.
\item \label{item:algo1} Using a standard procedure \cite{MR682664} for finding approximate integer relations between these vectors. We have used the LLL implementation in \texttt{sage} in our algorithm. We describe how to do this below.
\item \label{item:algo2} Verify that any approximate identities are in fact proper identities. In our case, we have used exact algebraic arithmetic in \texttt{sage} for this purpose on noting that all numbers in the vectors $\CTB_b(r,a)$ have coordinates which are arguments of algebraic numbers.
\item Output all relations found in this way.
\item Check for linear dependence among the obtained formulas to ensure that they are properly distinct. This is done by checking whether the Gram matrix of any selection of vectors is zero.
\end{enumerate}

The procedure in \cite{MR682664} describes how to find approximate integer relations between a set of real numbers. What we really need is an approximate relation between vectors, but in our case, this is the same thing if we work instead with the BBP-type formulas. Indeed, we do the following. First, we let $x_i = \BBP(1, r^{-b},b, \CTB_b(r,i))$ for $i=1, \dots, \frac{b}2-1$, and take suitably good rational approximations $\hat{x}_i$ to these numbers. We now pick a large integer $N$ and let $m_i = [N \hat{x}_i]$, the nearest integer to $N \hat{x}_i$. Consider now the lattice $\Lambda \subseteq \mathbb{R}^{\frac{b}2}$, spanned by the vectors $v_i = (0, \dots, 0, 1, 0, \dots, m_i)$ with $1$ in the $i$'th coordinate and $m_i$ in the $n+1$'st coordinate. Applying LLL to this lattice yields a non-zero vector of small norm. But any vector of the lattice $\Lambda$ is of the form
\[
\sum_{i=1}^{b/2 - 1} a_i v_i = \left(a_1, \dots, a_{b/2-1}, \sum_{i=1}^{b/2 - 1} a_i m_i\right).
\]
This vector is short only if the final coordinate is small, that is, if we have a near integer relation between the $m_i$.

We would hope that the relation is in fact an integer relation between the $x_i$, i.e. that $\sum_{i=1}^{b/2 - 1} a_i x_i = 0$. This is not an unreasonable hope. Indeed, if we let $\epsilon > 0$ be the worst among the qualities of our approximations $\hat{x}_i$, then
\begin{alignat*}{2}
N \left\vert \sum_{i=1}^{b/2 - 1} a_i x_i \right\vert &= \left\vert \sum_{i=1}^{b/2 - 1} a_i N(x_i - \hat{x}_i + \hat{x}_i) \right\vert \\ &\le N \epsilon \left\vert \sum_{i=1}^{b/2 - 1} a_i \right\vert + \left\vert \sum_{i=1}^{b/2 - 1}  a_i N \hat{x}_i \right\vert \\
& \le N \epsilon \left\vert \sum_{i=1}^{b/2 - 1} a_i \right\vert + \left\vert \sum_{i=1}^{b/2 - 1}  a_i m_i \right\vert + \frac{\left\vert \sum_{i=1}^{b/2 - 1}  a_i \right\vert}{2}.
\end{alignat*}
Letting 
\[
s = \left\Vert \sum_{i=1}^{b/2 - 1} a_i v_i \right\Vert^2
\]
the square of the Euclidean norm of the vector found by LLL, this immediately implies that 
\[
\left\vert \sum_{i=1}^{b/2 - 1} a_i x_i \right\vert \le \frac{s}{N} + \vert s \vert \epsilon.
\]
Note that by \cite{MR682664}, the number $s$ is within a constant factor of the square of the smallest non-zero Euclidean norm occurring for a vector in $\Lambda$. In other words, for $\epsilon$ small and $N$ large, we will expect the the left hand side to be really small and hence equal to zero to within some tolerance. Whether this approximate equality is in fact an equality is subsequently checked in the next step of the algorithm.

All that remains in order to include formulas for other numbers than zero is to include them among the $x_i$. If one is interested in obtaining a formula for $\pi$, say, one would just include $\pi$ in the list of $x_i$ as the $b/2$'th element, so that the lattice $\Lambda$ will now lie in $\mathbb{R}^{b/2+1}$. This will allow the algorithm to look for formulas for linear combinations including the new number as well, and if one was looking for integer combinations of other constants, these could also be added. Of course, it is entirely possible that the integer relations coming out of the search algorithm do not include the constants with a non-zero coefficient, in which case we have just recovered a zero formula. In fact, if one plugs in an entirely arbitrary number, in all probability the number in question does not admit a BBP-type formula at all, and thus will be ignored by the algorithm (though its inclusion may increase run time). In our search, we have looked for formulas for $\pi$ and multiples thereof. This is due to the argument below.

We now finally describe how to reduce a candidate identity to a problem in exact algebraic arithmetic, cf. Step \ref{item:algo2} of the algorithm. As described here, the procedure works for $\pi$.  

A candidate identity obtained in Step \ref{item:algo1} has the form 
\[
\sum_{i=1}^N a_i \arg\left(q+r e^{i\theta_i}\right) = b \pi,
\]
where the $a_i$ and $b$ are integers, not all zero; $N$ is the length of the formula, we are searching for (in our case $b/2-1$ or $b/2$); $r$ is algebraic; and the $\theta_i$ are rational multiples of $\pi$. Thus, $c_i = q+r e^{i\theta_i} \in \overline{\mathbb{Q}}$. Thus, rewriting and using properties of the argument,
\[
b\pi = \sum_{i=1}^N a_i \arg(c_i) = \sum_{i=1}^N \arg\left(c_i^{a_i}\right) = \arg\left(\prod_{i=1}^N c_i^{a_i}\right).
\]
Since Step \ref{item:algo1} outputs an integer $b$, this implies that the number inside the argument must be a real number, whence it must be equal to its complex conjugate. The $a_i$ are likewise integers, so the equality to be verified exactly is
\[
\prod_{i=1}^N c_i^{a_i} = \prod_{i=1}^N \overline{c_i}^{a_i},
\]
an identity in algebraic numbers, which can be verified in exact algebraic arithmetic.

An implementation of the above algorithm in \texttt{sage} can be found on \url{https://projects.au.dk/ada/}

The algorithm described above outputs any BBP-type zero formulas found in this way. As we will see below, it does not recover all known formulas. Nonetheless, we believe that the method has its merits. One such feature is the easily understood geometrical interpretation of the algorithm. Another is the fact that $r$ need not be an integer, but only an algebraic number. In our experiments, we have attempted to look for formulas with a Pisot number base. 

Another feature of the algorithm is that the format of the output in fact constitutes a geometric proof of the formula in question. This is to be understood in the sense that the output is a linear combination of BBP-type formulas given in terms of the $\CTB$-function. In order to remove any computer interaction from the proof of a given formula, one only needs to verify that the associated triangles have a good configuration. This is best illustrated with an example. The null formula
\begin{equation}
\label{eq:lafont}
\begin{split}
0 = 2^{12} \big(&\BBP(1,2^{12}, 24, \CTB_{24}(1/\sqrt{2}, 5))\\ &-  \BBP(1,2^{12}, 24, \CTB_{24}(1/\sqrt{2}, 11))\big),
\end{split}
\end{equation}
was originally found by Lafont in \cite{lafont}, though in the usual BBP-notation. Out algorithm recovers this formula. The form of the formula above allows us to turn the formula above into an easily verifiable geometric condition as illustrated in Figure \ref{fig:null-formula}. The large triangle gives rise to the first term, while the smaller triangle gives rise to the second term. The proof of the formula is simply a matter of verifying the fact that the complex numbers $0$, $1+2^{-1/2}e^{\frac{5}{12}\pi i}$ and $1+2^{-1/2}e^{\frac{11}{12}\pi i}$ are collinear -- essentially an exercise for undergraduates.

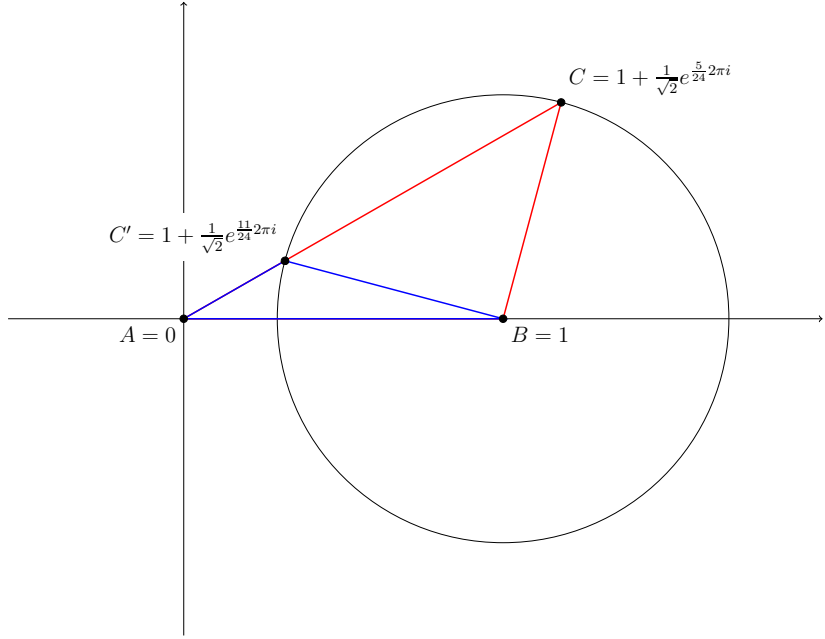
\begin{figure}
\scalebox{.7}{\begin{tikzpicture}
\draw[->] (-1.1*3,0) -- (4*3,0);
\draw[->] (0,-2*3) -- (0,2*3);
\draw (2*3,0) circle [radius=1.4142*3];
\draw[red, thick] (0,0) -- (7.09076, 4.098076) -- (6,0) -- (0,0);
\draw[blue, thick] (0,0) -- (1.901924, 1.098076) -- (6,0) -- (0,0);
\draw[fill=black] (0,0) circle (2pt) node[anchor=north east] {$A = 0$};
\draw[fill=black]  (7.09076, 4.098076) circle (2pt) node[anchor=south west] {$C = 1+ \tfrac{1}{\sqrt{2}}e^{\frac{5}{24} 2 \pi i}$};
\draw[fill=black]  (1.901924, 1.098076) circle (2pt) node[anchor=south east, fill=white] {$C' = 1+ \tfrac{1}{\sqrt{2}}e^{\frac{11}{24} 2 \pi i}$};
\draw[fill=black]  (1.901924, 1.098076) circle (2pt);
\draw[fill=black] (6,0) circle (2pt) node[anchor=north west] {$B=1$};
\end{tikzpicture}}
\caption{Geometric proof of Lafont's null formula.}
\label{fig:null-formula}
\end{figure}

\section{Computational results}

We consider first integer bases, which are powers of $2$ and $3$. In Table \ref{table:zero_formulas1}, we give a formula count of the known formulas in various bases and lengths with integer coefficients. In addition to the formulas with integer coefficients, our algorithm also finds formulas for $0$ with non-integer coefficients. These are omitted from the counts. All formulas obtained are of degree $1$. We list the values of base, the associated $r$ and the length $b$ as they occur in Theorem \ref{thm:main}.

\begin{table}
\begin{tabular}{|l|l|l|l|l|}
\hline
base & $r$ & $b$ & \# known formulas & \# recovered formulas \\
\hline
$2^4$ & $1/\sqrt{2}$ & $8$ & $1$ & $0$ \\
$2^6$ & $1/2$ & $6$ & $1$ & $0$ \\
$2^{12}$ & $1/\sqrt{2}$ & $24$ & $5$ & $1$ \\
$2^{20}$ & $1/\sqrt{2}$ & $40$ & $3$ & $1$ \\
$3^6$ & $1/\sqrt{3}$ & $12$ & $2$ & $1$ \\
\hline
\end{tabular}
\caption{Null formulas in integer bases}
\label{table:zero_formulas1}
\end{table}

In our condensed notation using the $\CTB$-notation which allows for easy geometric proof of the formulas, we have obtained the following known formulas aside from Lafont's formula \eqref{eq:lafont} discussed above.
\begin{equation*}
\begin{split}
0 = 2^{20} \big(& \BBP(1,2^{20}, 40, \CTB_{40}(1/\sqrt{2},3)) \\
&+\BBP(1,2^{20}, 40, \CTB_{40}(1/\sqrt{2},5)) \\
&+\BBP(1,2^{20}, 40, \CTB_{40}(1/\sqrt{2},11)) \\
&-\BBP(1,2^{20}, 40, \CTB_{40}(1/\sqrt{2},13)) \\
&-\BBP(1,2^{20}, 40, \CTB_{40}(1/\sqrt{2},15)) \\
&+\BBP(1,2^{20}, 40, \CTB_{40}(1/\sqrt{2},19)) \\
&-\BBP(1,2^{20}, 40, \CTB_{40}(1/\sqrt{2},20))\big),
\end{split}
\end{equation*}
which one finds in \cite{bailey}; and
\begin{equation}
\label{eq:zero}
\begin{split}
0 = 386\sqrt{3} \big(& \BBP(1,3^6, 12, \CTB_{12}(1/\sqrt{3},3)) \\
&- \BBP(1,3^6, 12, \CTB_{12}(1/\sqrt{3},5))\big),
\end{split}
\end{equation}
found by Adegoke in \cite{MR2740581}. The leading coefficients, which admittedly look a little silly, are there to ensure that if one rewrites the obtained formula as a single BBP-type formula, it becomes a pure formula without leading coefficient and with integer coefficients. For instance, on expanding and regrouping terms, \eqref{eq:zero} becomes
\[
0 = \BBP(1,729, 12, (243,-243,-324,-81,27,0,-9,9,12,3,-1,0)).
\]

Using our algorithm, we have found a BBP-type formula for $\sqrt{3}\pi$, which we have not been able to find in litterature. The formula reads 
\[
\sqrt{3}\pi = \frac{9}{2^{12}}\BBP(1,2^{12},24, A),
\]
with 
\begin{equation*}
\begin{split}
A = (&2^{11},0,0,2^{10}, 2^{9}, 0, 2^8,2^8, 0,0,2^6,0,\\&-2^5,0,0,-2^4, -2^3, 0, -2^2, -2^2,0,0,-1,0).
\end{split}
\end{equation*}
The computer generated proof found by the algorithm is
\begin{equation*}
\begin{split}
\sqrt{3}\pi  = 3 \sqrt{3} \big(& \BBP(1,2^{12}, 24, \CTB_{24}(1/\sqrt{2},5)) \\
&+ \BBP(1,2^{12}, 24, \CTB_{24}(1/\sqrt{2},11))\big). 
\end{split}
\end{equation*}
The BBP-type formula is found by expanding the right hand side.

The real merit of our algorithm is that it is able to produce BBP-type formulas to non-integer bases, while still producing formulas with integer coefficients. We have not been able to find examples of this in literature, though Adegoke \cite{MR3214376} does have BBP-type formulas in base $\phi$, the Golden ratio. However, the formulas obtained by Adegoke have coefficients in $\mathbb{Z}[\phi]$, which to us seems a little unnatural. The bases studied must be algebraic. We have made our algorithm run with several Pisot bases. Of course, the output will always be expansions with rational coefficients as is the case with integer bases, and as such they are not $\beta$-expansions in the sense of Rényi \cite{MR97374} and Parry \cite{MR142719}. Nevertheless, we find the expansions of interest. We are particularly intrigued by the rather large number of expansions with the Golden Ratio as a base.

In Table \ref{table:formulas2}, we present the number of null formulas found to various Pisot bases of low degree and height. In each case, the minimal polynomial of the Pisot number; the approximate value of the Pisot number, i.e. the real root which is  greater than one $1$, is given; the denominator $b$ is given; and the number of null formulas found are all given. The value of $r$ is in each case set equal to the reciprocal of the Pisot number in question. The values of $b$ chosen are $60$ and $24$ due to the smoothness of these numbers. The effect is that our algorithm finds formulas coming from even divisors of the numbers. 

\begin{table}
\begin{tabular}{|l|l|l|l|}
\hline
Polynomial & Approximate value& $b$ & \# formulas obtained  \\
\hline
$x^2-x-1$ & $1.618033$ & $60$ & $12$ \\
$x^2-2x-2$ & $2.732050$ & $60$ & $1$ \\
$x^2-2x-1$ & $2.414213$ & $60$ & $0$ \\
$x^2-3x+1$ & $2.618033$ & $60 $ & $4$ \\
$x^2-3x-1$ & $3.302775$ & $60$ & $0$ \\
$x^2-4x+2$ & $3.414213$ & $60$ & $0$ \\
$x^3-x-1$ & $1.3247179$ & $60$ & $3$ \\
$x^4 - x^3 - 1$ & $1.380277569$ & $24$ & $1$\\
\hline
\end{tabular}
\caption{Null formulas in Pisot bases}
\label{table:formulas2}
\end{table}

The first two lines of Table \ref{table:formulas2} are the Golden Ratio and $\sqrt{3}+1$ respectively. We denoting the Golden ratio by $\phi$, and putting $\psi = \sqrt{3}+1$, we present a small selection of the formulas obtained by our algorithm. For $\phi$, we have for instance the following formulas. Since searching for BBP-type formulas in non-integer bases appear to be largely unexplored territory, we have not found any of these in literature.
\begin{equation*}
\begin{split}
0  = - \big(& \BBP(1,\phi^{60}, 60, \CTB_{60}(\phi^{-1},18)) \\
&+\BBP(1,\phi^{60}, 60, \CTB_{60}(\phi^{-1},24)),
\end{split}
\end{equation*}
\begin{equation*}
\begin{split}
0  = - \big(& \BBP(1,\phi^{60}, 60, \CTB_{60}(\phi^{-1},14)) \\
&+\BBP(1,\phi^{60}, 60, \CTB_{60}(\phi^{-1},26)),
\end{split}
\end{equation*}
\begin{equation*}
\begin{split}
0  = - \big(& \BBP(1,\phi^{60}, 60, \CTB_{60}(\phi^{-1},8)) \\
&+\BBP(1,\phi^{60}, 60, \CTB_{60}(\phi^{-1},28)),
\end{split}
\end{equation*}
and
\begin{equation*}
\begin{split}
0  = - \big(& \BBP(1,\phi^{60}, 60, \CTB_{60}(\phi^{-1},8)) \\
&-\BBP(1,\phi^{60}, 60, \CTB_{60}(\phi^{-1},9)) \\
&+\BBP(1,\phi^{60}, 60, \CTB_{60}(\phi^{-1},21)).
\end{split}
\end{equation*}
In addition, our algorithm found a formula for for $\pi$ to base $\phi$,
\[
\pi = 10 \BBP(1,\phi^{15}, 15, \CTB(\phi^{-1},2)).
\]
Note that the $b$ in the latter formula is reduced to $15$. This reduction of the base is possible by the above remarks on the smoothness of $60$.

For $\psi =  \sqrt{3}+1$, the following null formula occurs in base $12$, even though we performed our search in base $60$,
\begin{equation*}
\begin{split}
0  = \big(& \BBP(1,\psi^{12}, 12, \CTB_{12}(\psi^{-1},2)) \\
&-\BBP(1,\psi^{12}, 12, \CTB_{12}(\psi^{-1},5)).
\end{split}
\end{equation*}
While this is the only null formula, which we have found, a formula for $\pi$ did show up,
\[
\pi = 12 \BBP(1,\psi^{12}, 12, \CTB_{12}(\psi^{-1},2)).
\]
Of course, by the two preceding formulas, it immediately follows that 
\[
\pi = 12 \BBP(1,\psi^{12}, 12, \CTB_{12}(\psi^{-1},5)).
\]

As a final remark on the non-integer formulas obtained, for the number $\phi+1$ with minimal polynomial $x^2-3x+1$, our usual approximation parameter, measured by $N$ in the application of LLL, of $N=10^{10}$ only resulted in two null formulas. Increasing the parameter to $N=10^{15}$ resulted in an additional two formulas. This suggests that we may be missing formulas, which should nevertheless show up as arising from triangles in the way described in this paper,  due to the restrictions in \texttt{sage}.

\section{Concluding remarks}

We end this paper with some concluding remarks and open problems.

First, the abundance of zero formulas found for the Golden Ratio is somewhat of a mystery! We really do not know why this particular irrational base yields so many zero formulas. Judging from the other Pisot numbers studied, it appears not to have anything to do with the Pisot'ness of the number. Of course, it is entirely possible that we have just not found all zero formulas for other numbers. The issue definitely requires further study.

Second, the original intention behind BBP-type formulas was to find deep digits in the expansions of numbers such as $\pi$, and indeed the original BBP formula allows one to calculate the $n$'th binary digit of $\pi$ without knowing all the preceding ones. For irrational bases, such as the Pisot bases studied in this paper, one also has a digital expansion, namely the $\beta$-expansions of Rényi  \cite{MR97374} and Parry \cite{MR142719}. It is not clear to us how to extract digital information from the expansions produced by our setup. Naively trying to extend the spigot algorithms to the irrational setup fails miserably. Obtaining digital information from our formulas remain an open problem.

Third and finally, we are only able to deal with BBP-type formulas of degree one in the present setup. Obtaining a similarly intuitive and geometrical approach to formulas of higher degree also remains an open problem. It would appear that the simple interpretation with triangles with vertices at $0$, $1$ and a point on a regular polygon on a particular circle is no longer appropriate for this type of problem. We leave it for the interested reader to ponder how to extend the interpretation of the formulas and their algorithmical implementation to the higher degree case.

\end{document}